\documentclass[a4paper, 11pt]{amsart}
\usepackage{amssymb,latexsym}
\usepackage{color}
\usepackage{graphicx}
\usepackage{epigraph}
\usepackage{appendix}
\usepackage{hyperref}
\usepackage[left=25mm, right= 25mm, top=20mm, bottom=20mm, includefoot, includehead]{geometry}
\theoremstyle{plain}
\newtheorem{theorem}{Theorem}[section]
\newtheorem{corollary}[theorem]{Corollary}
\newtheorem{proposition}[theorem]{Proposition}
\newtheorem{lemma}[theorem]{Lemma}
\newtheorem{conj}[theorem]{Conjecture}
\theoremstyle{definition}
\newtheorem*{term}{Terminology}
\newtheorem{definition}[theorem]{Definition}

\newtheorem{remark}[theorem]{Remark}

\theoremstyle{plain}
\newtheorem*{theo}{Theorem}

\usepackage{tcolorbox}
\newtcolorbox{mybox}{colback=blue!5!white,colframe=blue!75!black}

\newcommand{\enm}[1]{\ensuremath{#1}}          %
\newcommand{\CC}{\enm{\mathbb{C}}}

\newcommand{\NN}{\enm{\mathbb{N}}}

\newcommand{\ZZ}{\enm{\mathbb{Z}}}

\newcommand{\PP}{\enm{\mathbb{P}}}

\renewcommand{\phi}{\varphi}
\renewcommand{\theta}{\vartheta}
\renewcommand{\epsilon}{\varepsilon}

\newcommand{\old}[1]{}

\makeatletter
\@namedef{subjclassname@2020}{%
	\textup{2020} Mathematics Subject Classification}
\makeatother

\date{}


\title{On the codimension of permanental varieties}

\author{A. Boralevi, E. Carlini, M. Micha\l{}ek and E. Ventura}

\address{Politecnico di Torino, Dipartimento di Scienze Matematiche ``G. L. Lagrange'', Corso Duca degli Abruzzi 24\\
10129 Torino, Italy}
\email{ada.boralevi@polito.it}
\email{enrico.carlini@polito.it}
\email{emanuele.ventura@polito.it}

\address{Universit\"at Konstanz}
\email{mateusz.michalek@uni-konstanz.de}

\begin{document}

\begin{abstract}
In this article, we study {\it permanental varieties}, i.e., varieties defined by the vanishing 
of permanents of fixed size of a generic matrix. Permanents and their varieties play an important, and sometimes poorly understood, role in combinatorics. However, there are essentially no geometric results about them in 
the literature, in very sharp contrast to the well-behaved and ubiquitous case of determinants and minors. Motivated by the study of the singular locus of the permanental hypersurface, we focus on the codimension of these varieties. 
We introduce a $\CC^{*}$-action on matrices and prove a number of results. In particular, we improve a lower 
bound on the codimension of the aforementioned singular locus established by von zur Gathen in 1987. 
\end{abstract}

\maketitle

\begin{quote}
\leftskip=70pt {\it To our dear friend and colleague Gianfranco Casnati}
\end{quote}

\section{Introduction}
\epigraph{{\it Despite their structural similarity, the determinant and the permanent are \textnormal{worlds apart}.}}{Avi Wigderson \cite{W19}}

\noindent The most important polynomial associated to a square matrix is its determinant. In hindsight, its ubiquitous presence in mathematics might be related to its very large isotropic group, whose description dates back to Frobenius. Arguably, the second most important polynomial is the {\it permanent}. 
Let $M=(x_{i,j})$ be a $k\times k$ matrix with entries in a field $F$. Let $\mathfrak S_k$ be
the symmetric group of permutations of the set $\lbrace 1,\ldots, k\rbrace$. The {\it permanent} of $M$ is the polynomial
\[
\mathrm{perm}(M) = \sum_{\sigma\in \mathfrak S_k} x_{1,\sigma(1)}\cdots x_{k,\sigma(k)}. 
\]
The permanent has a much smaller isotropic group than the determinant, namely the product of the normalizers of two algebraic tori. Permanents and determinants are famously related by a generating function, the content of the {\it MacMahon's master theorem} \cite[vol. I, \S3, Chapter II, 63-66]{MacMah}. The striking tension between these two polynomials is at the heart of {\it geometric complexity theory}. Indeed, while the determinant may be computed in polynomial time using Gaussian elimination, the permanent is {\it not known} to be exactly computable in polynomial time. The mere existence of such an algorithm would imply {\bf P}={\bf NP}. The grand idea of geometric complexity theory is to approach fundamental problems in complexity theory using tools and techniques from algebraic geometry and representation theory. 
For instance, the ${\bf VP}$ versus ${\bf VNP}$ problem, that may be considered the polynomial cousin of the well-known  ${\bf P}$ versus ${\bf NP}$ problem, concerns finding a sequence $(p_k)_{k\in \NN}$ of polynomials whose algebraic circuit size grows faster than any polynomial in $k$; see \cite[\S 1.2]{Landsberg17} and \cite[\S 12.4]{W19} for details. Valiant conjectured that the $k\times k$ permanents $\mathrm{perm}_k$ form such a sequence \cite{Valiant79}. A possibly weaker but more concrete version of Valiant's conjecture deals with the complexity measure of the permanent, as opposed to the determinant. In detail, the {\it determinantal complexity} of a polynomial $p$ is the smallest number $\mathrm{dc}(p)$ such that
$p$ is an affine linear projection of a determinant of that size. Valiant conjectured that $\mathrm{dc}(\mathrm{perm}_k)$ grows faster than any polynomial in $k$ \cite{Valiant79}. The best result known so far is due to Mignon and Ressayre \cite{MR2005}: $\mathrm{dc}(\mathrm{perm}_k)\geq O(k^2)$; this first super-linear lower bound is still far from the full conjecture. 
On the algebraic geometry side, a study to compare the structure of Fano schemes of determinantal and permanental hypersurfaces was conducted by Chan and Ilten \cite{CI}.

It is worth noticing that permanents naturally arise in combinatorics and especially in graph theory \cite{Minc}. Given a bipartite graph $G$, one naturally associates to $G$ its adjacency square $0/1$-matrix $M_G$; then the permanent of $M_G$ is the number of perfect matchings of $G$. One difficult problem about $0/1$-matrices was posed by Minc in 1967, who asked 
for an upper bound on the value of the permanent. This was solved in 1973 by Br\'{e}gman \cite{Breg} and later by Radhakrishnan using {\it entropy} from quantum information theory \cite{Radha}. 
Permanents were also the subject of the Van der Waerden's conjecture for doubly stochastic matrices, that asked for a lower bound on the value of the permanent for such matrices. 
This was solved by Egorychev and Falikman in 1981 \cite{Eg,Fal}, and later also by Gurvits in 2007 \cite{Gur}, with a shorter argument involving stability of real polynomials. All these results have brilliant proofs and are nicely featured in the beautiful book by Aigner and Ziegler \cite{AZ}. Other interesting appearances of permanents in the sciences include applications to order statistics (the Bapat-Beg Theorem) \cite{BB89,Han94} and quantum mechanics, see \cite{CDM22} and references therein. 

Determinant and permanent share a common historically important generalization in representation theory. Given a $k\times k$ matrix $M$ and 
a partition $\lambda = (\lambda_1,\ldots,\lambda_s)$ of $k$, the {\it immanant} of $M$
is 
\[
\mathrm{Imm}_{\lambda}(M) = \sum_{\sigma\in \mathfrak S_k} \chi_{\lambda}(\sigma) x_{1,\sigma(1)}\cdots x_{k,\sigma(k)},
\]
where $\chi_{\lambda}$ is the character of the irreducible representation of $\mathfrak S_k$ corresponding to $\lambda$. The determinant and permanent 
are special immanants corresponding respectively to the alternating ($\lambda=(1,\ldots,1)$) and trivial ($\lambda=(k)$) representations. Immanants were 
introduced by Littlewood and Richardson \cite{LR34}. The problems we will tackle for permanents are interesting for every immanant, 
as not much is known about their geometry. It would be interesting to study their structural properties as $\lambda$ varies. 

Going back to permanents, it is apparent from the above discussion that they have a tendency to be extremely difficult objects to study. The perspective from which we look at them once again confirms this. Given a generic $k\times k$ matrix $M$ over a field $F$, with $k\geq 3$, the {\it permanental hypersurface} is $P = \lbrace \mathrm{perm}(M)=0\rbrace \subset F^{k\times k}$. A folklore question asks for a description of the singular locus of this hypersurface \cite{Landsberg17, VG}.
Similarly, one may ask this for any {\it immanantal hypersurface} and its singular locus. However this problem seems very elusive and much more involved than the corresponding one for the determinant already for the permanent. We expect that many of the techniques introduced in this paper for permanents carry 
over to study immanants, but the conclusions will depend on $\lambda$. 

\begin{definition}[{\bf Permanental rank \cite{YY}}]
Let $M$ be a matrix. Its {\it permanental rank} is the largest integer $k$ such that there is a $k\times k$ submatrix of $M$ whose permanent is nonzero. The permanental rank of $M$ is denoted $\mathrm{prk}(M)$. 
\end{definition}

\begin{term}
In this paper, by a {\it variety} over a field $F$, we mean a separated scheme of finite type 
that is reduced but not necessarily irreducible over $F$. Our varieties are affine cones over projective varieties that are typically reducible; we neglect the scheme structure of their components, as we shall be concerned with their codimension.
\end{term}

\noindent Given the affine cone $P_{k,n} = \lbrace \mathrm{prk}(M)\leq k-1\rbrace \subset F^{k\times n}$, 
we sometimes look at $\PP(P_{k,n})\subset \PP^{kn-1}$, its corresponding projective variety.
Since we are interested in the codimension of these varieties, we shall jump back and forth between
affine and projective spaces, according to the convenience of the approach at hand.\\

\noindent {\bf Main results.}\\
%Our main results are in the intertwined directions. 
We first study the codimension of the variety of maximal permanents of a generic matrix, in some ranges. 

\begin{theo}[Theorem \ref{codims}]
Let $F$ be a field of characteristic zero, and $M$ a generic $k\times n$ matrix of linear forms, with $n\geq k+1$. 
Then, for $2\leq k\leq 5$, the codimension of the variety $P_{k,n} = \lbrace \mathrm{prk}(M)\leq k-1\rbrace \subset F^{k\times n}$ is $n$. In particular, when $2\leq k\leq 4$, $P_{k,k+1}$ is a complete intersection. 
\end{theo}

\noindent We speculate (Conjecture \ref{conj: ci}) that the previous result holds in much more generality. In fact, the core of the proof of Theorem \ref{codims} for those special values of $k$ is based on the following. 

\begin{theo}[Theorem \ref{thm:codimP_kn}]
Let $k\geq 1$. 
If $P_{h,h+1}\subset F^{h\times (h+1)}$ has codimension $h+1$ for any $h\leq k$, then $P_{k,n}$
has codimension $n$, for any $n\geq k+1$. In particular, the validity of Conjecture \ref{conj: ci} for every $k\in \NN$ implies 
that $P_{k,n}$ has codimension $n$, for every $k\in \NN$ and $n\geq k+1$.
\end{theo}

\noindent Note that the sequence of ideals $I(P_{k,n})$ for $n\in \NN_{\geq k+1}$ is an example of {\it symmetric wide-matrix variety} of Draisma-Eggermont-Farooq \cite{DEF}. They show that the number of components up
to the action of the symmetric group $\mathfrak S_n$ is a quasi-polynomial in $n$ \cite[Theorem 1.1.1]{DEF}. 

We introduce a $T=\CC^{*}$-action on matrices which unravels a subtle geometric structure of $P_{k,k+1}$. We  establish a {\it correspondence} between certain vector bundles coming from the tangent bundle and irreducible components of $P_{k,k+1}$; see \S\ref{subsec:correspondence1} and \S\ref{subsec:correspondence2}. 

\begin{theo}[Theorem \ref{prop:irred comps of k x k+1}]
Let $X$ be any irreducible component of $Y=P_{k,k+1}$. 
Then $X$ coincides with $T^1_{X^T,Y}$, the closure of the total space of the weight one subbundle (under the torus action) of the tangent bundle over some open set in $X^T$. 
\end{theo}

\smallskip

In a second part of the paper we improve a lower bound due to von zur Gathen, established in 1987, again with the help of a $\CC^{*}$-action. This allows us to have a description for the irreducible components of the singular locus $\mathrm{Sing}(P)$
of the permanental hypersurface as subvarieties of total spaces of the aforementioned vector bundles; see \S\ref{subsec:singlocus}. 

Our result in this direction reads as follows. 

\begin{theo}[Theorem \ref{codim of sing loc}]
Let $k\geq 6$, and let $P=\lbrace \mathrm{perm}(M)=0\rbrace \subset \CC^{k\times k} $ be the permanental hypersurface. The codimension of the singular locus $\mathrm{Sing}(P) = \lbrace \mathrm{prk}(M)\leq k-2\rbrace$ satisfies the inequality $6 \leq \mathrm{codim} \: \mathrm{Sing}(P) \leq 2k$. 
\end{theo}
\vspace{3mm}

\noindent {\bf Organization of the paper.}\\
In \S \ref{sec:2xn}, we discuss the case of permanents of $2\times n$ matrices,
where we also underline the differences between our case and the case of minors. 
In \S\ref{sec:kxk+1}, we initiate the study of the variety of maximal permanents of $k\times (k+1)$ matrices. Its analysis will be developed further in \S\ref{sec:torus action}, where we employ
$\CC^{*}$-actions that are useful to organize irreducible components of the aforementioned variety. 
In these two sections we prove the main results showcased in this introduction. Finally, in \S\ref{codes}, we include the scripts used to deal with the description of irreducible components in \S\ref{subsec45} and \S\ref{subsec56}.\\

\noindent {\bf Acknowledgements.}\\
A.B., E.C. and E.V. are members of GNSAGA group of INdAM (Italy). M.M. was supported by the DFG grant
467575307. E.V. was partially supported by 
the INdAM-GNSAGA project ``Classification Problems in Algebraic Geometry: Lefschetz Properties and Moduli Spaces'' (CUP\textunderscore E55F22000270001). 
E.V. also thanks the Department of Mathematics of Universit\"at Konstanz, for warm hospitality and financial support, where part of this project was conducted.
This study was carried out within the ‘0-Dimensional Schemes, Tensor Theory, and Applications’ project 2022E2Z4AK – funded by European Union – Next Generation EU  within the PRIN 2022 program (D.D. 104 - 02/02/2022 Ministero dell’Universit\`a e della Ricerca).
M.M. thanks the Institute for Advanced Study for a great working environment and support through the Charles Simonyi Endowment.
We thank an anonymous reviewer for useful suggestions improving the readability, and for providing the majority of the newer code in \S\ref{codes}, which simplifies that in an earlier version of this paper. We acknowledge the invaluable help of the algebra software \texttt{Macaulay2} \cite{M2}. 

\section{Permanents of $2\times n$ matrices of linear forms}\label{sec:2xn}

In this section $F$ will be a field of characteristic zero, unless explicitly stated otherwise. The ideal generated by $2\times 2$ permanents of a generic matrix is by now very much understood. A Gr\"{o}bner basis and a complete description of its minimal primes were obtained in \cite{LS2000}. More recently, 
in \cite{GHSW}, the authors determined the minimal free resolution of the $2\times 2$ permanents of a $2\times n$ matrix. The results and observations in this section are most naturally stated in projective space.

\begin{theorem}\label{theo: 2xn}
Let $M$ be a generic $2\times n$ matrix of linear forms with $n\geq 3$. 
The variety $\PP(P_{2,n})= \PP(\lbrace \mathrm{prk}(M)\leq 1\rbrace) \subset \PP^{2n-1}$ 
has codimension $n$. Its singular locus has dimension $1$ and consists of $n^2$ lines.  
\begin{proof}
We first assume that the entries of $M$ are the coordinates of $\PP^{2n-1}$ denoted by $x_{ij}$, where $1\leq i\leq 2$ and $1\leq j\leq n$. %The irreducible components of $\PP(P_{2,n})$ are two $\PP^{n-1}$'s and $\binom{n}{2}$ smooth quadrics in  $\PP^3$. They are all smooth. 
%For $n=3$, $\PP(P_{2,n})$ is defined either by the vanishing of a row (this gives two $\PP^2$'s) or by the vanishing of the $j$-th column and the smooth quadric $\mathrm{perm}(M_{[3]\setminus \lbrace j\rbrace, [3]\setminus \lbrace j\rbrace})$. 
By \cite[Theorem 4.1]{LS2000}, the matrices in $\PP(P_{2,n})$ are such that either $(n-2)$ of their columns vanish and the smooth quadric (on the left two columns) vanishes, or one of their rows vanishes. Thus, the variety has codimension $n$. The irreducible components of $\PP(P_{2,n})$ are two $\PP^{n-1}$'s and $\binom{n}{2}$ quadrics in  $\PP^3$. They are all smooth. 

We introduce $n^2$ lines as follows: for each $x_{ij}$, consider the $n$ lines whose local coordinates
are $x_{ij}$ and $x_{\ell k}$ where either $\ell=i$ and $k\in [n]\setminus \lbrace j\rbrace$ or $\ell\neq i$ and $k=j$. 

We show that these $n^2$ lines are in the singular locus $S = \mathrm{Sing}(\PP(P_{2,n}))$. 
To see this, up to permuting rows or columns, we have two types of lines: 
$r_1$ with local coordinates $x_{11}, x_{12}$, and $r_2$ with local coordinates $x_{11}, x_{21}$. 

Let $J = [n]\setminus \lbrace 1,2\rbrace$. The line $r_1$ is contained in one of the two $\PP^{n-1}$'s and the smooth quadric in a $\PP^3$ defined by the ideal $(x_{1, j\in J}, x_{2, j\in J}, x_{11}x_{22}+x_{12}x_{21})$. Therefore $r_1$ is in $S$. The line $r_2$ is contained in $(n-1)$ of the smooth quadrics above whose equation is $x_{11}x_{2j}+x_{21}x_{1j}=0$ for $j=2,\ldots, n$. Hence $r_2$ is in $S$. There are $n^2-n$ lines of the same type as $r_1$, and $n$ of the same type as $r_2$. 
In conclusion, $S$ contains the $n^2$ lines just described. 

Now we look at all the set-theoretic intersections of the irreducible components. The two copies of $\PP^{n-1}$ are disjoint. 
Each $\PP^{n-1}$ intersects all the $\binom{n}{2}$ smooth quadrics in $(n^2-n)/2$ lines (all these lines are of the same type as $r_1$). Two quadrics intersect at most along one of the lines of the same type as $r_2$. Thus $S$ is contained in the $n^2$ lines described above. 

Any matrix $M'$ (regarded as a vector) in the orbit under the linear action of $\mathrm{GL}(2n,\CC)$ of the vector $M\in \CC^{2n}$ is a matrix with $2n$ linearly independent linear forms. The induced  action on projective space preserves the invariants of the irreducible components of $\PP(P_{2,n})$ and $S$. 
\end{proof}
\end{theorem}

\begin{remark}
Let $M$ be a $2\times n$ Hankel matrix, i.e. a matrix of the form
\begin{equation}\label{Hankel}
M = \begin{pmatrix} 
x_0 & x_1 & \cdots & x_{n-1} \\
x_1 & \cdots & x_{n-1} & x_n \\
\end{pmatrix}.
\end{equation}

While $\PP(\lbrace \mathrm{rk}(M)\leq 1\rbrace)\subset \PP^n$ is a rational normal curve of degree $n$, the permanental version (in characteristic different than $2$) is different and surprisingly small, as shown in the next result.
\end{remark}

In the next proposition, we are interested in the scheme structure. 

\begin{proposition}\label{hankel-2xn}
Let $F$ be a field, $\mathrm{char}(F)\neq 2$, and let $M$ be a $2\times n$ Hankel matrix of the form \eqref{Hankel}. The scheme $\PP(P_{2,n})=\PP(\lbrace \mathrm{prk}(M)\leq 1\rbrace)\subset \PP^n$ is zero-dimensional and of degree $8$, supported at two points. In particular, the degree of this zero-dimensional scheme does not depend on $n$. If $\mathrm{char}(F)=2$, then $\PP(P_{2,n}) = \PP(\lbrace \mathrm{rk}(M)\leq 1\rbrace)$ is a rational normal curve of degree $n$.  
\end{proposition}
\begin{proof}
From the description given in \cite{LS2000}, it is easy to see that the scheme is supported on the two points 
$p_0 = [1:0:\cdots:0]$ and $p_n = [0:\cdots:0:1]$. We work on the affine chart $\{x_n\neq 0\}$. In this chart, call the resulting ideal $J_{n}$. 
Then, for $2\leq i\leq n$, we have $x_{n-i} + x_{n-i+1}x_{n-1}\in J_n$. 

Let 
\[
g_1 = -x_{n-2}^2-x_{n-3}x_{n-1}-x_{n-2}x_{n-1}+x_{n-1}^2-x_{n-3}-\frac{1}{2}x_{n-2}, 
\]
\[
g_2=-x_{n-2}^2-x_{n-3}x_{n-1}+x_{n-1}^2+x_{n-2}-\frac{1}{2}x_{n-1},
\]
\[
g_3 = x_{n-2}x_{n-1}+x_{n-1}^2+x_{n-3}+x_{n-2}+\frac{1}{2}.
\] 
Then $x_{n-1}^4 = g_1(x_{n-1}^2 + x_{n-2}) + g_2 (x_{n-1}x_{n-2}+x_{n-3})   + g_3(x_{n-2}^2 + x_{n-1}x_{n-3})\in J_n$. 
Hence, for $4\leq i\leq n$, the variable $x_{n-i}$ is zero in the quotient $K[x_1,\ldots, x_n]/J_n$. Thus $1, x_{n-3}, x_{n-2}, x_{n-1}$ generate the quotient and form a basis. It follows that the scheme has degree $4$ at $p_n$. By symmetry, it has degree $4$ at $p_0$, as well. 
\end{proof}

\section{Permanents of $k\times (k+1)$ matrices of linear forms}\label{sec:kxk+1}

\begin{proposition}\label{lower-upperboundscodim}
Let $F$ be an arbitrary field, $k\geq 2$, and $M$ a generic $k\times n$ matrix of linear forms, with $n\geq k$. 
Then the codimension of $P_{k,n} := \lbrace \mathrm{prk}(M)\leq k-1\rbrace \subset F^{k\times n}$ 
satisfies the inequality $n-k+1\leq \mathrm{codim}(P_{k,n})\leq n$. 
\begin{proof}
The upper bound $\mathrm{codim}(P_{k,n})\leq n$ holds for any $n\geq k\geq 2$, because we have linear spaces
of codimension $n$ inside $P_{k,n}$.  

For the lower bound, we proceed by induction on $k\geq 2$. The case $k=2$ is settled in Theorem \ref{theo: 2xn}. Let $k\geq 3$, and assume that the statement is true for $k-1$. Let $C$ be an irreducible component of $P_{k,n}$. We have two cases: 
either $C$ is a cone over an irreducible component of $P_{k-1,n}(M')$ where $M'$ (up to permuting rows and columns)
is a $(k-1)\times n$ matrix of linear forms, or it is not. In the first case, by induction $C$ has codimension at least $n-(k-1)+1=n-k+2>n-k+1$ in $F^{k\times n}$. In the second case, let $A$ be the generic point of $C$ and set $J=\lbrace 1,\ldots, k-1\rbrace$. We may assume $\mathrm{perm}(M_{J,J})(A)\neq 0$, i.e. the $(k-1)\times (k-1)$ permanent of the upper-left corner of $M$ is nonzero at $A$. Thus, for all $k\leq j\leq n$, $x_{{k,j}_{|C}}$ is a rational function on $C$ in the rational functions $x_{i,\ell}$ where $1\leq i\leq k-1$ and $\ell\in [n]$ or $i=k$ and $1\leq \ell\leq k-1$. 
So we have an inclusion of function fields $F(C)\subset F(x_{i,\ell})$, where $x_{i,\ell}$ are the $(k-1)(n+1)$ coordinates above. Hence $\dim_{F^{k\times n}} C = \mathrm{transdeg}_F(F(C))-1\leq (k-1)(n+1)-1$ and then $\mathrm{codim}(C)\geq (kn-1)-(k-1)(n+1)-1 = n-k+1$.  
\end{proof} 
\end{proposition}

\begin{corollary}\label{cor:perm is irreducible}
Let $F$ be an arbitrary field, $k \geq 2$, and $M$ a generic $k\times k$ square matrix over $F$. Let $P = \lbrace \mathrm{perm}(M)=0\rbrace$ be the permanental hypersurface, and denote by $\mathrm{Sing}(P)$ its singular locus. Then $\mathrm{codim} \ \mathrm{Sing}(P)\geq 4$. In particular, $\mathrm{perm}(M)$ is an irreducible polynomial over $F$. 
\begin{proof}
The singular locus of the permanental hypersurface is 
\[
\mathrm{Sing}(P) = P_{k-1,k} = \lbrace \mathrm{prk}(M)\leq k-2\rbrace. 
\]
A similar strategy as in the proof of Proposition \ref{lower-upperboundscodim} shows
that $\mathrm{codim}(P_{k-1,k})\geq 4$. If $\mathrm{perm}(M)$ were reducible over $F$, 
then the codimension of $\mathrm{Sing}(P)$ would be at most $2$. Thus $\mathrm{perm}(M)$ is irreducible over $F$. 
\end{proof}
\end{corollary}

\begin{lemma}\label{lem:perms are l.i.}
Let $M$ be a generic $k\times n$ matrix and let $h\leq \min\lbrace k, n\rbrace$. Then the 
$h\times h$ permanents of $M$ are linearly independent. 
\begin{proof}
Every such permanent is of the form $\mathrm{perm}(M')$ for some $h\times h$ submatrix $M'$ of $M$. Hence it is uniquely determined by the monomial given by the product of the elements in the main diagonal of $M'$.  This monomial does not appear in any other $\mathrm{perm}(M'')$ for $M''\neq M'$. 
\end{proof}
\end{lemma}

Proposition \ref{lower-upperboundscodim} applies to ideals of maximal minors as well. In fact, it is very weak 
when $n=k+1$. In contrast, we propose the following

\begin{conj}\label{conj: ci}
Let $M$ be a generic $k\times (k+1)$ matrix with $k\geq 2$. Then $P_{k,k+1}$ is a complete intersection. In particular, $\mathrm{codim}(P_{k,k+1}) = k+1$. 
\end{conj}
This conjecture holds true for $k=2$. One can show the following result. 

\begin{proposition}\label{algebraicindependence}
Let $M$ be a generic $k\times (k+1)$ matrix with $k\geq 1$. Then the $k\times k$ permanents of $M$ are algebraically independent. 
\begin{proof}
The statement for $k=1$ is obvious. Fix $k\geq 2$, let $M = (x_{ij})$ and $N$ be the $(k+1)\times (k+1)$ matrix obtained from $M$ by adding a row of $(k+1)$ extra variables $y_{1},\ldots, y_{k+1}$.
Let $P_N = \mathrm{perm}(N)$ be the permanent polynomial of $N$ in the $(k+1)^2$ variables $x_{ij}, y_{\ell}$. Then the $k\times k$ permanents of $M$ are the $(k+1)$ partial derivatives $\partial P_N/\partial_{y_{\ell}}$ of $P_N$. The main result in \cite{MR2005} is proven showing that the Hessian matrix of $P_N$ has nonzero determinant. This is equivalent (see \cite[\S 7]{libro_Russo}) to saying that the first partial derivatives of $P_N$ are algebraically independent. Hence any subset of first partial derivatives of $P_N$ consists of algebraically independent elements. 
\end{proof}
\end{proposition}

\begin{remark}
By Proposition \ref{algebraicindependence}, there is no analogue in the permanental case of Pl\"{u}cker coordinates of minors of a $k\times (k+1)$ matrix. However, permanents are generally not algebraically independent. For instance, one can check that 
the $10$ permanents of a $2\times 5$ generic matrix are algebraically dependent. 
\end{remark}

Next, we prove that the linear spaces 
given by the vanishing of a row are indeed irreducible components of $P_{k,k+1}$. Due to unmixedness, to prove Conjecture \ref{conj: ci} it would be enough to show that $P_{k,k+1}$ is arithmetically Cohen-Macaulay. 

\begin{remark}
In the case of ideals of minors of fixed size $k$ of any generic matrix, one way to show that they are Cohen-Macaulay ideals is to show that the corresponding quotient ring is of the form $S^G\subset S$, where $S$ is a polynomial ring and $S^G$ is the subring of invariants under the action of the group $G=\mathrm{GL}(k,F)$. This approach cannot work in the case of permanents, because we know that $P_{k,k+1}$ must be reducible by the next proposition. 
\end{remark}

\begin{proposition}\label{linearspaces}
The variety $P_{k,k+1}\subset F^{k\times (k+1)}$ has at least $k$ linear spaces among its irreducible components in codimension $k+1$. 
\begin{proof}
Let $M$ be a $k\times (k+1)$ of the form
\[
M = \begin{pmatrix}
x_{1,1} & x_{1,2} & \ldots & x_{1,k+1} \\
x_{2,1} & x_{2,2} & \ldots & x_{2, k+1} \\
\vdots & \vdots & \vdots & \vdots  \\
x_{k,1} & x_{k,2} & \ldots & x_{k,k+1} \\
\end{pmatrix}.
\]
Consider the projective variety $\PP(P_{k,k+1})\subset \PP^{k(k+1)-1}$. We show that the $k$ linear spaces in $\PP(P_{k,k+1})$ defined by the vanishing of one row of $M$ are irreducible components of $\PP(P_{k,k+1})$. 

Up to the action of the symmetric group permuting rows, it is enough to show that the linear space $L_k$ whose defining ideal is $(x_{k,1},\ldots, x_{k,k+1})$ is an irreducible component of $P_{k,k+1}$.  

We fix a point $p\in \PP(P_{k,k+1})$ whose coordinates are $x_{i,j}(p)=1$ for all $1\leq i\leq k-1$ and $1\leq j\leq k+1$, and zero otherwise. We work inside the affine chart $U = \lbrace x_{1,1}\neq 0\rbrace\cong F^{k(k+1)-1}$ of $\PP^{k(k+1)-1}$.  
Inside $U$, we change coordinates so that $p$ is the origin of $U$. In this coordinates $\widetilde{x}_{i,j}$, the variety $P_{k,k+1}\cap U$ is defined by ideal $\widetilde{I}$ generated the $k\times k$ permanents of the following $k\times (k+1)$ matrix 
\[
\widetilde{M} = \begin{pmatrix}
1 & \widetilde{x}_{1,2}+1 & \ldots & \widetilde{x}_{1,k+1}+1 \\
\widetilde{x}_{2,1}+1 & \widetilde{x}_{2,2}+1 & \ldots & \widetilde{x}_{2, k+1}+1 \\
\vdots & \vdots & \vdots & \vdots  \\
\widetilde{x}_{k,1} & \widetilde{x}_{k,2} & \ldots & \widetilde{x}_{k,k+1} \\
\end{pmatrix}.
\]
%Let $\widetilde{I}^{\mathrm{in}}$ be the ideal generated 
%by the homogeneous components of the lowest degree of all $g\in \widetilde{I}$. 
Let $\widetilde{I}^{\mathrm{lin}}$ be the ideal generated by the linear part of all $g\in \widetilde{I}$. 
T%he affine tangent cone $C_p(\PP(P_{k,k+1}))$ to $p$ at $\PP(P_{k,k+1})$ and t
he affine tangent space $T_p(\PP(P_{k,k+1}))$ to $p$ at $\PP(P_{k,k+1})$ %are
is given by %$\mathrm{Spec}(F[\widetilde{x}_{i,j}, (i,j)\neq (1,1)]/\widetilde{I}^{\mathrm{in}})$
%and by 
$\mathrm{Spec}(F[\widetilde{x}_{i,j}, (i,j)\neq (1,1)]/\widetilde{I}^{\mathrm{lin}})$%, respectively
. %Let $f\in \widetilde{I}$ be a polynomial (not necessarily homogeneous). Then 
%\[
%f = g_1h_1 + \cdots + g_{k+1}h_{k+1},
%\]
%where 
Let $g_j$ for $j=1,\dots,k+1$ %are
be the permanents of $\widetilde{M}$% and 
%the $h_j$ are polynomials in $F[\widetilde{x}_{i,j}, (i,j)\neq (1,1)]$
. 
The linear part of $g_j$ is% of the form 
\[
%\sum_{i=1}^{k+1}c_i\cdot (k!)\cdot (\widetilde{x}_{k,1}+\widetilde{x}_{k,2}+\cdots + \widehat{\widetilde{x}_{k,i}} + \cdots + \widetilde{x}_{k,k+1}), \ \ c_i\in F,
\widetilde{x}_{k,1}+\widetilde{x}_{k,2}+\cdots + \widehat{\widetilde{x}_{k,i}} + \cdots + \widetilde{x}_{k,k+1},
\]
where %the $c_i$ are terms of degree $0$ of the $h_i$, and 
$\widehat{\widetilde{x}_{k,i}}$ means that we are omitting that summand. Since we are in characteristic zero, these linear forms are %all nonzero and 
linearly independent. 

Hence, the tangent space $T_p(\PP(P_{k,k+1}))$ is of dimension at most the dimension of $L_k$. As $L_k\subset \PP(P_{k,k+1})$ the tangent space has to be of the same dimension as $L_k$. Thus, $p$ is a smooth point of $\PP(P_{k,k+1})$ belonging to a unique component of dimension equal to the dimension of $L_k$. It follows that $L_k$ is an irreducible component of $ \PP(P_{k,k+1})$.
%So we have two cases. If the linear part of $f$ is not zero, 
%then it is contained in the ideal $J= (\widetilde{x}_{k,1}+x_{k,2}+\cdots + \widehat{\widetilde{x}_{k,i}} + \cdots + x_{k,k+1}, 1\leq i\leq k+1)$. Note that $J = (\widetilde{x}_{k,1},\ldots, \widetilde{x}_{k,k+1})$. The discussion above implies $\widetilde{I}^{\mathrm{lin}} = J$. 
%
%
%Suppose the linear part of $f$ is zero. Note that 
%every monomial in each $g_j$ is divisible by at least one of 
%the variables generating  $\widetilde{I}^{\mathrm{lin}}$. 
%Hence the same is true for each of the monomials in the smallest degree part of $f$. This implies $\widetilde{I}^{\mathrm{in}} = \widetilde{I}^{\mathrm{lin}}$. 
%Hence $p$ is a smooth point of $P_{k,k+1}$. Then we have the following chain of equalities
%\[
%(k-1)(k+1)-1 = \dim C_p(\PP(P_{k,k+1})) = \dim T_p(\PP(P_{k,k+1})) = \dim (\PP(P_{k,k+1}))_p,
%\]
%where the right-most number is the local dimension of the variety $P_{k,k+1}$ at $p$. Since the local dimension of $\PP(P_{k,k+1})$ at $p$
%equals the dimension of the linear space $L_k$, the latter is an irreducible component of $\PP(P_{k,k+1})$. 
\end{proof}
\end{proposition} 

\begin{remark}\label{CompCol}
Other types of irreducible components arise
from the vanishing of the generic linear forms in a column. One can check that they all have codimension $k+1$ as well. 
\end{remark}

\begin{proposition}\label{cases:3x4 and 4x5}
For $k=3$ or $4$, let $M$ be a generic $k\times (k+1)$ matrix of linear forms. Then the codimension of $P_{k,k+1} = \lbrace \mathrm{prk}(M)\leq k-1\rbrace \subset F^{k\times (k+1)}$ is $k+1$. Equivalently, Conjecture \ref{conj: ci} holds for $2\leq k\leq 4$ and $P_{k,k+1}$ is a complete intersection. 
\begin{proof}
Let $k=3$. Define $L$ to be the linear space
transforming the matrix $M=(x_{ij})$ into the following {\it circulant Hankel matrix}
\[
H_3 = \begin{pmatrix}
x_{11} & x_{12} & x_{13} & x_{14} \\
x_{12} & x_{13} & x_{14} & x_{15} \\
x_{13} & x_{14} & x_{15} & x_{11} \\
\end{pmatrix}.
\]
Let $V_3 = P_{3,4}\cap L$. The ideal $I_{V_3}$ of $V_3$ is the ideal of $3\times 3$ permanents of $H_3$. Using \texttt{Macaulay2}, we check that $\mathrm{ht}(I_{V_3})=4$. Thus $\dim(Z)= 1$ for all irreducible components $Z\subset V_3$. Therefore,
for any irreducible component $X\subset P_{3,4}$, we have
\[
1=\dim(Z) \geq \dim(X) + \dim(L) - 12 = \dim(X) -7. 
\]
Thus $\dim(X)\leq 8$ and hence $\mathrm{codim}(P_{3,4})\geq 4$. 

Let $k=4$. Define $L$ to be the linear space
transforming the matrix $M=(x_{ij})$ into the following {\it circulant Hankel matrix}
\[
H_4 = \begin{pmatrix}
x_{11} & x_{12} & x_{13} & x_{14} & x_{15} \\
x_{12} & x_{13} & x_{14} & x_{15} & x_{11} \\
x_{13} & x_{14} & x_{15} & x_{11} & x_{12} \\
x_{14} & x_{15} & x_{11} & x_{12} & x_{13} \\
\end{pmatrix}. 
\]
Let $V_4 = P_{4,5}\cap L$. The ideal $I_{V_4}$ of $V_4$ is the ideal of $4\times 4$ permanents of $H_4$. Using \texttt{Macaulay2}, we check that $\mathrm{ht}(I_{V_4})=5$. Thus $\dim(V_4)=0$. On the other hand, for any irreducible component $X\subset P_{4,5}$, we have 
\[
0=\dim(V_4) \geq \dim(X) + \dim(L) - 20 = \dim(X) - 15. 
\]
Thus $\dim(X)\leq 15$ and hence $\mathrm{codim}(P_{4,5})\geq 5$.\end{proof}
\end{proposition}

Although we could have employed \texttt{Macaulay2} to compute
directly the codimension of $P_{3,4}$ and $P_{4,5}$, we believe that restricting to a suitable linear space (or, more generally, to a variety) might be a strategy to give the desired lower bound on the codimension. In fact, the approach pursued in Proposition \ref{cases:3x4 and 4x5} comes from the observation that the dimension of the ideal of permanents of a circulant Hankel matrix tends to be small. This could certify the codimension of the original permanental ideal. The behaviour is clear for $2\times 2$ permanents, as shown in the following lemma. 

\begin{lemma}\label{2x2_circulantHankel}
Let $k\geq 2$ and let $S=F[x_{j}]$ be the polynomial ring in the $k+1$ variables $x_{j}$ with $1\leq j\leq k+1$. 
Let $M=(x_{j})$ be a $k\times (k+1)$ circulant Hankel matrix. Let $Q_1=\lbrace \mathrm{prk}(M)\leq 1\rbrace$
be the variety whose ideal is generated by the $2\times 2$ permanents of $M$. Then $\mathrm{codim}(Q_1) = k+1$. 
\begin{proof}
For $k=2$, the statement can be checked similarly as in the proof of Proposition \ref{hankel-2xn}. 
Let $k\geq 3$. It is enough to check that some power of each $x_j$ is in the ideal $I(Q_1)\subset S$. First note that, by definition, all the generators of $I(Q_1)$ are of the form
\[
x_{\ell'}x_{m'} + x_{\ell}x_{m}, \mbox{ where } \ell'+ m'\equiv \ell + m  \mod k+1. 
\]
Vice versa, any equality of the form $\ell'+ m'\equiv \ell + m  \mod k+1$ gives rise to a generator of $I(Q_1)$. To see these statements, note that giving a generator of $I(Q_1)$
is equivalent to giving an arbitrary choice of two rows and two columns. Fix a row $r$, pick two elements on $r$, say $x_{\ell}$ and $x_{\ell'}$ where $\ell'\equiv \ell+h \mod k+1$; we have just selected two columns. Now, choose a second row $r'$, pick the elements $x_{m'}$, where $m'\equiv\ell+s \mod k+1$, and $x_{m}$ where $m\equiv\ell+s+h \mod k+1$; these last two choices are forced as we have already selected the two columns. Notice that the indices of the variables satisfy the desired equation in modular arithmetic. 

For each $1\leq j\leq k+1$, we have distinct generators of the form 
\begin{equation}\label{rel1}
x_j^2 + x_{\ell}x_{m}, 
\end{equation}
\begin{equation}\label{rel2}
x_j^2 + x_{\ell'}x_{m'}, \mbox{ and}
\end{equation}
\begin{equation}\label{rel3}
x_{\ell}x_{m} + x_{\ell'}x_{m'},
\end{equation}
where $\ell, m\neq j$ are possibly the same index; similarly $\ell',m'\neq j$ are possibly the same index. 
Thus $1/2\cdot \eqref{rel1} + 1/2\cdot \eqref{rel2} - 1/2\cdot \eqref{rel3} = x_j^2\in I(Q_1)$ for every $1\leq j\leq k+1$. 
\end{proof}
\end{lemma}

\begin{remark}
In comparison, the codimension of the ideal of $2\times 2$ permanents of a generic $k\times (k+1)$ matrix is $k^2-1$ \cite[Theorem 4.1]{LS2000}. 
\end{remark}

\subsection{The nondegenerate permanental ideal}\label{subsec:kirkup ideal}
Kirkup shows that $\mathrm{codim}(P_{3,4})=4$ \cite[\S 7]{K}. For each $k$, he looks at the colon ideal $J_k = I(P_{k,k+1}):(\prod_{i,j}x_{i,j})^{\infty}\supset I(P_{k,k+1})$ \cite[Corollary 10]{K}. This ideal %complicates and at the same times enlivens 
unravels a lot of the structure of $I(P_{k,k+1})$, therefore it deserves a definition on its own. 

\begin{definition}
The ideal $J_k = I(P_{k,k+1}):(\prod_{i,j}x_{i,j})^{\infty}\supset I(P_{k,k+1})$ is called the {\it nondegenerate permanental ideal}. The corresponding variety $\mathcal V_k = V(J_k)$ is the {\it nondegenerate permanental variety}. 
\end{definition}

\begin{proposition}[Kirkup]
The nondegenerate permanental variety $\mathcal V_3$ is irreducible of codimension $4$ and degree $66$.  
\end{proposition}

The next lemma, which is implicit in \cite[\S 6]{K}, shows that the nondegenerate permanental variety is non-empty and distinct from the linear spaces of Proposition \ref{linearspaces} or from the 
irreducible components arising from the vanishing of a single column (Remark \ref{CompCol}). 

\begin{lemma}[Kirkup]\label{kirkupmatrices}
Let $k\geq 3$. The {\it Kirkup matrix}
\[
\mathcal K_{k,k+1} := \begin{pmatrix}
1 & 1 & 1 & 1 & 1 & 2-3k \\
\vdots & \vdots & \vdots & \vdots & \vdots & \vdots \\
\vdots & \vdots & \vdots & 1 & 1 & 2-3k \\
1 & 1 & \ldots & 1 & (2-2k) & (k-2)(k-1) \\
1 & 1 & \ldots & 1 & k & (2k-1)(k-2) \\
\end{pmatrix} \in \ZZ^{k\times (k+1)}
\]
is an element of the nondegenerate permanental variety: $\mathcal K_{k,k+1}\in \mathcal V_k$. 
\end{lemma}
\noindent It is indeed immediate to show that the Kirkup matrix has all $k\times k$ permanents vanishing. 

\begin{definition}[{\bf Kirkup components}]
A {\it Kirkup component} is an irreducible component of $\mathcal V_k$ containing the Kirkup 
matrix $\mathcal K_{k,k+1}$. 
\end{definition}

When $k=3$, $\mathcal V_3$ is also the unique Kirkup component. 
Kirkup found equations for $\mathcal V_k$ which we recall here. Let $M = (x_{i,j})$ be a $k\times (k+1)$ generic matrix and let $\mathrm{perm}_j(M)$ be the permanent of the matrix obtained from $M$ by removing the $j$th column. 
Define 
\[
M_{\ell,i,j} = \frac{\partial}{\partial x_{\ell,i}}\mathrm{perm}_j(M).
\]
Then $M_{\ell,i,j}$ is the $(k-1)\times (k-1)$ permanent of the matrix obtained from $M$ by omitting row $\ell$ and columns $i$ and $j$ (when $i=j$ the value is zero). We define two types of matrices: 
\[
A_{j} = \begin{pmatrix}
M_{1,1,j} & \ldots & M_{1,k+1,j} \\
\vdots & \vdots & \vdots \\
M_{k,1,j} & \ldots & M_{k,k+1,j} \\
\end{pmatrix}
\]
and 
\[
B_{\ell} = \begin{pmatrix}
M_{\ell,1,1} & \ldots & M_{\ell,1,k+1} \\
\vdots & \vdots & \vdots \\
M_{\ell,k+1,1} & \ldots & M_{\ell,k+1,k+1} \\
\end{pmatrix}. 
\]
The matrix $A_j$ is $k\times (k+1)$ and its $j$th column is zero; call $C_j$ the $k\times k$ submatrix obtained from $A_j$ by omitting its $j$th column of zeros. The matrix $B_{\ell}$ is symmetric of format $(k+1)\times (k+1)$. 

The next result is \cite[Proposition 9]{K}; we include a proof for completeness. 

\begin{proposition}[Kirkup]
The determinants $f_{j} = \det(C_j)$ and $g_{\ell} = \det(B_{\ell})$ are in the nondegenerate permanental ideal $J_k$. 
\begin{proof}
Thanks to the action of the symmetric group, 
it is enough to show the statement for $f_1$ and for $g_1$. For $j\neq 1$, by the Laplace expansion we have $\mathrm{perm}_j = \sum_{i=1}^k x_{i,1}M_{i,1,j}$. Let $e_j$ be the determinant of the $(k-1)\times (k-1)$ submatrix of $C_1$ obtained by omitting the first column and the $j$th row. Then
\[
I(P_{k,k+1})\ni \sum_{j=1}^k (-1)^j e_j\cdot \mathrm{perm}_j(M) = \sum_{i=1}^k\sum_{j=1}^k(-1)^j M_{i,1,j}\cdot e_j. 
\]
For $i=1$, the interior sum in the right-most side is $\det(C_1)=f_1$. For $i\neq 1$, the interior sum is the Laplace expansion of the determinant of the matrix obtained from $C_1$ replacing the first column with its $i$th column; so this sum is zero. Hence
$x_{1,1}\cdot f_1\in I(P_{k,k+1})$. Thus, by definition, $f_1\in J_k$. The proof for the $g_{\ell}$'s is similar, using the relation $\mathrm{perm}_j(M) = \sum_{i=1}^k x_{1,i} M_{1,i,j}$. 
\end{proof}
\end{proposition}

\subsection{Permanents of $k\times n$ matrices of linear forms}

\begin{theorem}\label{codims}
Let $F$ be a field of characteristic zero and $M$ a generic $k\times n$ matrix of linear forms, with $n\geq k+1$. 
The codimension of the variety $P_{k,n} = \lbrace \mathrm{prk}(M)\leq k-1\rbrace \subset F^{k\times n}$ is $n$ for $2\leq k\leq 5$. 
\begin{proof}
We perform induction on the number of columns $n$ and on the number of rows $k$. The base cases $2\times 3$, $3\times 4$, $4\times 5$, $5\times 6$ are proven in Theorem \ref{theo: 2xn}, Propositions \ref{cases:3x4 and 4x5} and \ref{5x6 case}. 

Let $n\geq k+2$ and suppose that the statement is proven for $n-1$ and any number of rows $\leq k-1$. Let $C$ be an irreducible and reduced component of $P_{k,n}$. We have two cases: either $C$ is a cone over an irreducible component of $P_{k-1,n}(M')$ where $M'$ is the $(k-1)\times n$ submatrix of $M$ consisting of the first $(k-1)$ rows, or it is not. In the first case, $C$ has codimension $\geq n$, because the only irreducible components of $P_{k-1,n}(M')$ have codimension $\geq n$ by inductive hypothesis.  

In the second case, up to permuting columns, we may assume that in $C$ the Zariski principal open set $U_C = C\cap \left\lbrace f\neq 0  \right\rbrace$, where $f=\mathrm{perm}(M_{[1,\ldots, k-1],[1,\ldots,k-1]})$, is nonempty. Let $B = M_{[1,\ldots, k],[1\leq j\leq n-1]}$ be the $k\times (n-1)$ matrix consisting of the first $n-1$ columns of $M$. We denote the linear forms in $B$ by $x_{ij}$. Let $z$ be the linear form in the $(k,n)$-th entry of $M$. For $1\leq i\leq k-1$ and $j=n$, let $y_{ij}$ be the linear form in the $(i,j)$-th entry of $M$. 

Define $R = F[x_{ij}, y_{ij}, z]$. Let $I\subset R$ be the ideal generated by the $k\times k$ permanents in $B$, and let $J\subset R$ be the ideal generated by the $k\times k$ permanent $\mathrm{perm}(M_{[1,\ldots,k], [1,\ldots,k-1,n]})$. The latter may be expressed as
\[
\mathrm{perm}(M_{[1,\ldots,k], [1,\ldots, k-1, n]}) = f\cdot z + g(x_{ij}, y_{ij}),
\]
where $g(x_{ij}, y_{ij})$ is some polynomial in the variables $x_{ij}, y_{ij}$ only. Note that $I+J\subset I(P_{k,n})\subset I(C)$, where the latter is the prime ideal of $C$. Let $Y$ be the affine variety defined by the ideal $I+J$. 
Therefore $\emptyset \neq U_C = C\cap \left\lbrace f\neq 0  \right\rbrace\subset Y\cap \left\lbrace f\neq 0  \right\rbrace = U_Y$. The inclusion implies $\dim(U_C)\leq \dim(U_Y)$. 
The coordinate ring of the principal Zariski open set $U_Y$ is the localization of the coordinate ring of $Y$ at the element $f\in F[x_{ij}]$, i.e. $F[U_Y] = R[f^{-1}]/(I'+J')$, where $I'$ and $J'$ are the ideals $I$ and $J$ defined above after 
localizing at $f$. 

Let $S = F[x_{ij},y_{ij}]$. Note that the rings $R[f^{-1}]$ and $S[f^{-1}]$ are domains and $J' = (z+h)$, for $h\in S[f^{-1}]$. 

We show that $F[U_Y]$ is isomorphic to the ring $S[f^{-1}]/\widetilde{I}$, where $\widetilde{I}$ is the ideal in $S[f^{-1}]$ generated by the $k\times k$ permanents of $B$. An element of $F[U_Y]$ is an equivalence class $\overline{g} = g+(z+h)g_1 + \sum_{i=1}^{\ell}p_i q_i$, where the $p_i$'s are the generators of $I$ 
and $q_i\in R[f^{-1}]$. Since the latter ring is a domain, we may perform
Euclidean division of $g$ and of each $q_i$ by the element $z+h$. So each equivalence class is of the form $\overline{g} = g + (z+h)g_2 + \sum_{i=1}^{\ell}p_i r_i$, where $\deg_{z}(g)=0$ and $\deg_{z}(r_i) = 0$ for each $i$. This condition means that the $g, r_i\in S[f^{-1}]$ for each $i$. Define the map 
\[
\phi: F[U_Y]\longrightarrow S[f^{-1}]/\widetilde{I}
\]
where $\phi(\overline{g}) = g + \sum_{i=1}^{\ell}p_i r_i$. This is a ring isomorphism. 

Now, regard the ideal $I$ above as an ideal in $S$. By induction, $\mathrm{ht}(I) = n-1$. %Note that $f\notin I(C)$ by assumption. 
%In particular, $f\notin p$ for some associated minimal prime $p$ of $I$. On the contrary, viewing the minimal primes $q_i$ of $I$ inside $R$ one would have the inclusion $f\in \bigcap_i q_i\subset I(C)$ because $I(C)$ is an ideal containing $I$, a contradiction. 
%Thus the image of each such prime $p$ in $S[f^{-1}]$ is a minimal associated prime of $\widetilde{I}$ and 
As height can only go up after localisation with respect to an element not in the ideal, one has $\mathrm{ht}(\widetilde{I})\geq n-1$.

Since $S[f^{-1}]$ is a finitely generated domain, one has the equality 
\[
\dim(F[U_Y]) = \dim(S[f^{-1}]/\widetilde{I}) =
\]
\[
=\dim(S[f^{-1}])-\mathrm{ht}(\widetilde{I})\leq (kn-1)-(n-1) = (k-1)n. 
\]
Thus $\dim(C) = \dim(U_C)\leq \dim(U_Y) \leq (k-1)n$ and hence
$\mathrm{codim}(C)\geq n$.\end{proof} 
\end{theorem}

In this last result we explicitly employed the knowledge of the codimension of $P_{k,k+1}\subset F^{k\times (k+1)}$ for $2\leq k\leq 5$. Notice that the core of the proof of Theorem \ref{codims} is the following. 

\begin{theorem}\label{thm:codimP_kn}
Let $k\geq 1$. 
If $P_{h,h+1}\subset F^{h\times (h+1)}$ has codimension $h+1$ for any $h\leq k$, then $P_{k,n}$
has codimension $n$, for any $n\geq k+1$. In particular, the validity of Conjecture \ref{conj: ci} for every $k\in \NN$ implies 
that $P_{k,n}$ has codimension $n$, for every $k\in \NN$ and $n\geq k+1$.
\end{theorem}

\section{Torus actions}\label{sec:torus action}

In this section, we work over the field of complex numbers. Let $T=\CC^{*}$ act on a vector space $V$ with weights $0$ and $1$. This means that $V=V_0\oplus V_1$ and for every $v\in V_i$ and $t\in T$, we have $t\cdot v = t^{i}v$. The action induces naturally an action on $\PP(V)$. 

Let $Y\subset V$ be a $T$-invariant 
variety. Then, since the torus is irreducible, any irreducible component $X$ is invariant under 
the action of $T$. Let $X^{T}$ be the locus of fixed points under the $T$-action. It is not difficult to check that $X^T$ is smooth if $X$ is so. Then there exists a morphism 
\[
\phi_{t\rightarrow 0}: X \longrightarrow X^{T},
\]
defined by $\phi_{t\rightarrow 0}(x) = \lim_{t\rightarrow 0} t\cdot x\in X^{T}$. This is the restriction of the projection $V\rightarrow V_0$ with kernel $V_1$. Suppose $X$ is an affine cone
over $X'\subset \PP(V)$. Then $\phi_{t\rightarrow 0}$ induces a map 
\[
\psi_{t\rightarrow 0}: U \longrightarrow X'^{T},
\]
where $U\subset X'$ is the set of all $[x]\in X'$ such that $\lim_{t\rightarrow 0} t \cdot [x] \neq 0$. 

For any $x\in X^T$, the tangent space $T_{Y,x}$ to $Y\supset X$ splits into two summands of weight $0$ and $1$, 
which we call $T^0_{Y,x}$ and $T^1_{Y,x}$, respectively. 
\begin{definition}
    Let $Y\subset V$ be a variety, where $T=\CC^*$ acts on $V$ as above. Let $Z$ be an irreducible subvariety of $Y$. There exists a nonempty, Zariski open subset $U\subset Z$, such that the restriction of the tangent sheaf of $Y$ to $U$ is a vector bundle $T_{Y|U}$. If $T$ acts on $U$ then we obtain $T_{Y|U}=T^1_{Y|U}\oplus T^0_{Y|U}$. The total spaces of all three bundles map naturally to $V$ and we identify $T^1_{Y|U}$ with its image.
    
    We define $T^1_{Z,Y}$ as the Zariski closure in $V$ of $T^1_{Y|U}$. This irreducible variety does not depend on the choice of $U$. 
\end{definition}

\begin{proposition}\label{bb in our case}
%Let $W\subset X$ be the set of $x\in X$ such that $\alpha_x=1$, i.e. the set of points of weight $1$. 
%Suppose $Z=\phi_{t\rightarrow 0}(W)$ is dense in $X_i^T$, an irreducible component of $X^T$. The tangent subbundle $T^1_{Y}$ naturally defines a vector bundle over the constructible set $Z\subset X_i^{T}$. Then $X$ is contained in the closure of $T^1_{Y|Z}$. 
%Let $U$ be a nonempty Zariski open subset of $X\setminus X^T$. 
%Suppose $Z=\phi_{t\rightarrow 0}(U)$ is dense in $X_i^T$, an irreducible component of $X^T$. The tangent subbundle $T^1_{Y}$ naturally defines a vector bundle over the constructible set $Z\subset X_i^{T}$. Then $X$ is contained in the closure of $T^1_{Y|Z}$.
Any irreducible component $X$ of a variety $Y\subset V$ as above is contained in $T^1_{X^T,Y}$.
\begin{proof}
%If $x\in X_i^T$, then $x\in \overline{Z}$, which is in the closure of $T^1_{Y|Z}$. Otherwise, 
%It is enough to prove that $W$ is contained in the closure of $T^1_{Y|Z}$. 
Note that $X^T$ is irreducible, as it is the image $\phi_{t\rightarrow 0}(X)$. If $X=X^T$, 
then the statement is true as the total space $T^1_{X^T,Y}$ is canonically identified with $X^T$, as each fiber is zero. If $X^{T}\neq X$, the general point $x\in X$ will map to a point $\phi_{t\rightarrow 0}(x)\in X^T$ belonging to an open set $U\subset X^T$ over which $T_Y$ is a vector bundle. We have to prove that $x\in T^1_{Y,\phi_{t\rightarrow 0}(x)}$. 
%Let $x\in U$ and hence $\phi_{t\rightarrow 0}(x)=x_0\in Z$. 
Indeed, the whole orbit $T\cdot x$ consists of vectors in the fiber $T^1_{Y,\phi_{t\rightarrow 0}(x)}$, as the closure of this orbit is a line contained in $Y$, passing through $\phi_{t\rightarrow 0}(x)$ and on which $T$ acts with weight one. This proves the desired inclusion. %Hence, $x$ belongs to the closure of $T^1_{Y|Z}$. Thus $U$ and $X$ is in the closure of $T^1_{Y|Z}$. 
\end{proof}
\end{proposition}
Proposition \ref{bb in our case} is a simpler version of the Bia\l{}ynicki-Birula decomposition theorem \cite[Chapter II, Theorem 4.2]{Carrell}, but in a possibly singular context.

\begin{definition}[{\bf Type}]
The {\it type} of $X$ is the rank of the bundle  $T^1_{Y|U}$ in Proposition \ref{bb in our case}.  
\end{definition}

\subsection{Irreducible components of $P_{k,k+1}$ and torus actions}\label{subsec:correspondence1}

Let $W$ be the linear component of $Y=P_{k,k+1}$ given by the vanishing of the first row. If $p\in W$, 
then the Jacobian $J(Y)_p$ of $Y$ at $p$ is a $(k(k+1))\times (k+1)$-matrix with two blocks: 
\[
J(Y)_p = \begin{pmatrix}
B_{1}(A_p) \\
{\bf 0} \\
\end{pmatrix},
\]
where $A_p$ is the $(k-1)\times (k+1)$ nonzero submatrix of $p$ and 
$B_{1}(A_p)$ is the matrix introduced in \S \ref{subsec:kirkup ideal} evaluated at the entries of $A_p$. The matrix $B_{1}(A_p)$ is a symmetric $(k+1)\times (k+1)$-matrix whose main diagonal consists of zeros. The zero-block ${\bf 0}$ has size $((k-1)(k+1))\times (k+1)$.  

Let $T=\CC^{*}$ act on $V=\CC^{k\times (k+1)}$ scaling by $t$ the first row of a matrix in $V$ and preserving the other entries. 

\begin{corollary}\label{cor:tang at P kx(k+1)}
For any irreducible component $X$ of $Y=P_{k,k+1}\subset V$, we have a map 
\[
\phi_{t\rightarrow 0}: X \longrightarrow X^T \subset W=V^T. 
\]
If $X\neq W$, then any point $p\in X^T$ is in the singular locus of $Y$. Let $A_p$ be the corresponding $(k-1)\times (k+1)$ nonzero submatrix of $p$. 
For $p\in \phi_{t\rightarrow 0}(X)$, its tangent space to $Y$ is 
\[
T_{Y,p} = V^T\oplus \ker(B_1(A_p)), 
\]
where $T^0_{Y,p}=V^T$ and $T^1_{Y,p}=\ker(B_1(A_p))$. In particular, $\dim_{\CC} T^1_{Y,p} = \dim_{\CC} \ker(B_1(A_p))$. 
\begin{proof}
The point $p$ has to be singular as it belongs to two components: $X$ and $W$. 
The tangent space $T_{Y,p}$ is the kernel of transpose of $J(Y)_p$. The kernel of the matrix $B_1(A_p)$ sits inside 
the span of the variables $x_{1,h}$ for $1\leq h\leq k+1$, corresponding to the first row, which is 
a complement to the subspace $V^T$ in $V$. 
\end{proof}
\end{corollary}

Let $\mathcal X_i = \lbrace p\in V^T \ | \ \mathrm{crk}(B_{1}(A_p)) = i \rbrace \subset V^T$ for $0\leq i\leq k+1$, where $\mathrm{crk}$ denotes the corank of $B_{1}(A_p)$, regarded as a matrix in $\CC^{(k+1)\times (k+1)}$. This is a constructible set. 

\begin{theorem}\label{prop:irred comps of k x k+1}
Let $X$ be any irreducible component of $Y=P_{k,k+1}$. % and let $Z=\phi_{t\rightarrow 0}(X)$. 
Then $X$ coincides with $T^1_{X^T,Y}$. Note that there is a unique $i$ such 
that the generic point of $X^T$ sits inside $\mathcal X_i$, i.e. $X$ is of type $i$. 
\begin{proof}
By Proposition \ref{bb in our case}, $X$ is a subvariety of $T^1_{X^T,Y}$. 
There is a unique such index $i$, by the irreducibility of $X^T$ and by semi-continuity of matrix rank. 

We have to show the opposite inclusion. By the irreducibility of $T^1_{X^T,Y}$ and since $X$ is closed, it is 
enough to show that $T^1_{Y|X^T\cap \mathcal X_i}\subset X$. Let $(q,p)\in T^1_{Y|X^T}$ for $p\in X^T\cap \mathcal X_i$. Hence 
$q\in \ker(B_1(A_p))$. Now, regard $q = (q_1,\ldots, q_{k+1})\in \CC^{k+1}$. 
By definition, one has
\[
B_1(A_p)\cdot 
\begin{pmatrix}
q_1 \\
\vdots \\
q_{k+1}\\
\end{pmatrix} = \begin{pmatrix}
0 \\
\vdots \\
0\\
\end{pmatrix}. 
\]
The $j$-th linear condition $B_1(A_p)_j \cdot q^{t} = 0$ is equivalent to the vanishing of the $k\times k$ permanent 
of the submatrix of $M$ obtained by removing the $j$-th column of $M$ and evaluating at $(q,p)$. Hence $(q,p)\in Y$ and thus $(tq,p)\in Y$ for every $t\in T$. Thus $T^1_{Y|X^T\cap \mathcal X_i}$ is irreducible, contained in $Y$ and intersecting $X$ in a Zariski dense set. As $X$ is a component of $Y$ it follows that $X=T^1_{X^T,Y}$.
%This shows the desired equality. 
\end{proof}
\end{theorem}

\begin{corollary}\label{cor: codim for k x k+1}
Let $X$ be any irreducible component of $Y$. Let %$Z = \phi_{t\rightarrow 0}(X)$
%and
$p\in X^T$ be general. Then 
\[
\dim X = \dim X^T + \dim_{\CC} \ker B_1(A_p). 
\]
Equivalently, one has 
\[
\mathrm{codim}\ X = \mathrm{codim}_{V^T} X^T + \mathrm{rk}(B_1(A_p)). 
\]
\end{corollary}

\subsection{Correspondence between vector bundles and components of $Y=P_{k,k+1}$.}\label{subsec:correspondence2}
Theorem \ref{prop:irred comps of k x k+1} shows a geometric feature lurking behind the variety 
$Y$, that is a hierarchy of irreducible components: each irreducible component corresponds to a vector bundle of rank $\mathrm{crk}(B_1(A_p))$. 
The irreducible component $W=V^T\subset Y$ is of type $0$: $\mathrm{codim}_{V^T} \phi_{t\rightarrow 0}(W)=0$ and $B_1(A_p)$ is full-rank for a general $p\in V^T$. The irreducible components that are cones over irreducible components of $P_{k-1,k+1}$ correspond to $\mathrm{rk}(B_1(A_p))=0$, 
and hence they are of type $k+1$. 

\begin{proposition}\label{prop: no component corresponding to rank 1}
There is no irreducible component $X$ of $Y$ such that the general point of $X^T$ belongs to $ \mathcal X_{k}$, i.e. there is no irreducible component of type $k$. 
\begin{proof}
A necessary condition for the existence of such an irreducible component is 
that $\mathcal X_{k}\neq \emptyset$. This condition means that $\mathrm{rk}(B_1(A_p))=1$ 
for some $p\in V^T$. However, $B_1(A_p)$ is a symmetric matrix with zeros on the diagonal. Therefore, if it is nonzero,
then it has a $2\times 2$ submatrix $N$ of the form 
\[
N = \begin{pmatrix}
0 & a \\
a & 0 \\
\end{pmatrix},
\]
with $a\neq 0$. Thus $\mathrm{rk}(B_1(A_p))\geq 2$, whenever $B_1(A_p)$ is nonzero. 
\end{proof}
\end{proposition}

\begin{proposition}
The irreducible components described in Remark \ref{CompCol}, i.e. those given by 
choosing a vanishing column and the vanishing complementary permanent, are of type $k-1$. 
\begin{proof}
Given such an irreducible component, it is immediate to see that 
$B_1(A_p)$ has nonzero only one row and one column and so $\mathrm{rk}(B_1(A_p)) = 2$. 
\end{proof}
\end{proposition}

The locus $\mathcal X_1\subset V^T$ is given by the vanishing locus of $\det(B_{1})$ from which one 
removes the locus of larger corank. Hence this is a codimension-one constructible set inside $V^T$, 
which is possibly reducible. 

\begin{corollary}\label{cor: irred corresponding to line bundles}
Let $X$ be an irreducible component of $Y$ such that the general point of $X^T$ belongs to $ \mathcal X_1$,
i.e. $X$ is of type one. Then $\mathrm{codim}\ X \geq k+1$. 
\begin{proof}
By Corollary \ref{cor: codim for k x k+1}, we have $\mathrm{codim}\ X = \mathrm{codim}_{V^T} X^T+k$. Note
that we have $\mathrm{codim}_{V^T} X^T\geq \mathrm{codim}_{V^T} \mathcal X_1\geq 1$, which shows the inequality. 
\end{proof}
\end{corollary}

\begin{proposition}\label{prop: irred comp containing kirkup matrix}
Every irreducible component containing the Kirkup matrix $\mathcal K_{k,k+1}$ is of type one and so it has codimension at least $k+1$. 
\begin{proof}
Let $X$ be an irreducible component of $Y$ containing $\mathcal K_{k,k+1}$. %Let $Z=\phi_{t\rightarrow 0}(X)$ and l
Let 
$p_k$ be the corresponding point of $\mathcal K_{k,k+1}$. To prove the statement it is enough to show that $X^T\cap \mathcal X_1\neq \emptyset$ by Corollary \ref{cor: irred corresponding to line bundles}. To check the validity of the latter statement, it is sufficient to prove that $\mathrm{rk}(B_1(A_{p_k}))=k$. The upper-left $k\times k$ submatrix of $B_1(A_{p_k})$ is of the form 
\[
E=\begin{pmatrix}
0 & a & \cdots & a & b \\
a & 0 &  \cdots & a & b \\
\vdots & \vdots & \ddots & \vdots & \vdots \\ 
a & a & \cdots & 0 & b\\
b & b & \cdots & b & 0\\
\end{pmatrix},
\]
where $a,b\neq 0$. The main diagonal of $E$ consists of zeros. One has $\mathrm{rk}(E)=k$. 
To see this, first note that the $(k-1)\times (k-1)$ submatrix $A$ of $E$ only consisting of $a$'s
and zeros on the main diagonal has full-rank $k-1$: indeed the vector ${\bf 1}=(1,\ldots,1)\in \CC^{k-1}$ is in the row span of $A$ (add up all the rows and scale), and so every standard vector in $\CC^{k-1}$ is in its row span (subtract from ${\bf 1}$ each scaled row). Now, one can easily check that the last row of $E$ cannot be linearly dependent from the first $k-1$ rows. Thus $\mathrm{rk}(B_1(A_{p_k}))=k$. 
\end{proof}
\end{proposition}

\begin{remark}
    To understand dimensions of irreducible components of $P_{k,k+1}$ one simply needs to understand dimensions of irreducible components of each $\mathcal X_i$. Indeed, every irreducible component of $P_{k,k+1}$ comes from a rank $i$ vector bundle over a component of $\mathcal X_i$. However, it is not easy to conclude that components of $\mathcal X_i$ have codimension $i$ in $W$. As we have seen some $\mathcal X_i$ are empty, thus it is not true that $\mathcal X_{i+1}$ is contained in the closure of $\mathcal X_i$. 
\end{remark}

\subsection{Codimension of $P_{4,5}$}\label{subsec45}

We use the correspondence with vector bundles, to prove that the codimension of 
all irreducible components of $P_{4,5}$ is $5$. This gives an alternate proof 
of the case $k=4$ in Proposition \ref{cases:3x4 and 4x5}. 

\begin{center}
\begin{table}
    \begin{tabular}{|c|c|}
    \hline
        Type & Irreducible components\\
        \hline
       $0$  & $V^T$\\
         \hline
       $1$  & Kirkup component\\
         \hline
       $2$  & No components \\
         \hline
        $3$ & There exist such components\\
         \hline
        $4$ & No components\\
         \hline
         $5$ & Cones over $P_{3,5}$\\
         \hline
    \end{tabular}
   \caption{Irreducible components of $P_{4,5}$} \label{tab:4x5}
\end{table}
\end{center}

\begin{proposition}
All the irreducible components of $P_{4,5}$ have codimension $5$. In Table \ref{tab:4x5}, 
we organize them according to their type. 
\begin{proof}
We employ \texttt{Macaulay2} to perform the required computations. In \S\ref{codes}, we provide a script to check some of the cases reported in the table. 
For instance, the script checks that $\det B_1(A_p)$ is smooth in codimension one after intersecting scheme-theoretically with a subspace. As singular points remain singular after such intersection, this implies that $\det B_1(A_p)$ is smooth in codimension one. If it had several components, then each one would be of codimension one in $W$, and as all varieties we deal with are cones over projective varieties, the components would need to intersect in codimension one inside $\det B_1(A_p)$ \cite[Theorem 2.22]{MS21}. In particular, the variety would have to be singular in codimension one. Thus, we conclude that $\det B_1(A_p)$ is irreducible. 
This implies that there is a unique Kirkup component.

In a similar way, by intersecting $\mathcal X_{2}$ with a fixed linear subspace, the script verifies that there are no type $2$ irreducible components in codimension $\leq 5$. Since $I(P_{4,5})$ is generated by five polynomials, Krull's principal
ideal theorem implies that there are no type $2$ components. %Both computations are performed by intersecting the desired components with a fixed linear subspace and then computing the dimension.  
\end{proof}
\end{proposition}

\subsection{Components of $P_{5,6}$}\label{subsec56}

We use the correspondence with vector bundles to prove that the codimension of 
all the irreducible components of $P_{5,6}$ is $6$. 

\begin{proposition}\label{5x6 case}
All the irreducible components of $P_{5,6}$ have codimension $6$. In Table \ref{tab:5x6}, 
we organize them according to their type. 
\begin{proof}
We employ \texttt{Macaulay2} to perform the required computations. In \S\ref{codes}, we provide a script to check the case where $\mathrm{rk}(B_1(A_p))=2$, i.e. 
the irreducible components of type $4$. Let $X$ be an irreducible component 
of this type so we have $X^T\subset \mathcal X_4$, where 
$\mathcal X_4$ is defined by the $3\times 3$ minors of the matrix $B_1(A)$,
$A$ being a generic matrix in $V^T$. By Corollary \ref{cor: codim for k x k+1},
$\mathrm{codim} \ X = \mathrm{codim}_{V^T} X^T + \mathrm{rk}(B_1(A)) = \mathrm{codim}_{V^T} X^T + 2$. 
So $\mathrm{codim}\ X\geq 6$ is equivalent to verifying that 
$\mathrm{codim}_{V^T} X^T\geq 4$. Since $X^T\subset \overline{\mathcal X_4}$, it is sufficient to check 
that $\mathrm{codim}_{V^T} \mathcal X_4\geq 4$. Let $\PP(\overline{\mathcal X_4})\subset \PP(V^T)$ be the corresponding projective variety. Then it is enough 
to find a $L=\PP^3\subset \PP(V^T)$ such that their intersection $\PP(\mathcal X_4)\cap L$ is empty \cite[Theorem 2.22]{MS21}. 
The choice of such a suitable $L$ is reported on the script. The output of the script
reads: \texttt{[gb]{12}(400){13}(420){14}(840)number of (nonminimal) gb elements = 455,
number of monomials = 49455, used 33.8193 seconds}. With a similar code, we also check all the other cases. For instance, in type $3$, we find a $\PP^2$ such that $\PP(\mathcal X_3)\cap \PP^2\subset \PP(V^T)$ is empty. To check that in type $1$ we have a unique Kirkup irreducible component, we confirm that the singular locus of the set defined by $\det(B_1(A))=0$, for $A\in V^T$, has codimension higher than two in $V^T$. 
\end{proof}
\end{proposition}

\begin{center}
\begin{table}
    \begin{tabular}{|c|c|}
    \hline
        Type & Irreducible components\\
        \hline
       $0$  & $V^T$\\
         \hline
       $1$  & Kirkup component\\
         \hline
       $2$  &  No components\\
         \hline
        $3$ &  Potential components in codimension $6$\\
         \hline
         $4$ & There exist such components \\
         \hline
         $5$ & No components \\
         \hline
         $6$ & Cones over $P_{4,6}$ \\
         \hline
    \end{tabular}
   \caption{Irreducible components of $P_{5,6}$} \label{tab:5x6}
\end{table}
\end{center}

\subsection{Singular locus of the permanental hypersurface: von zur Gathen's problem}\label{subsec:singlocus}

Let $k\geq 3$, $M$ be a generic $k\times k$ square matrix, and let $P = \lbrace \mathrm{perm}(M)=0\rbrace$ be the $k\times k$ permanental hypersurface. A folklore question asks for a description of the singular locus of this hypersurface in terms of numerical invariants of various kinds. This is a challenging and poorly understood question, in sharp contrast with the singular locus of the determinantal hypersurface that has natural interpretation in terms of rank of matrices. 

The codimension of this set is currently unknown for $k\geq 5$. A first result towards determining its codimension, which was so far the strongest in this direction, is due to von zur Gathen: 

\begin{theorem}[von zur Gathen \cite{VG}]\label{vzgthm}
Let $k\geq 3$. The singular locus $\mathrm{Sing}(P) = \lbrace \mathrm{prk}(M)\leq k-2\rbrace$ has codimension between $5$ and $2k$. 
\end{theorem} 

Note that $Y=\mathrm{Sing}(P)$ is defined by the $(k-1)\times (k-1)$ permanents of $M$. Let $V$ be the vector space
of $k\times k$ complex matrices and let $X$ be an irreducible component of $Y$. We fix the following $T=\CC^{*}$-action 
on $V$: given $p\in V$, let $t\cdot p$ be the matrix where the first two rows are those of $p$ scaled by $t$, while the other entries are unchanged. Hence $V^T$ is the linear space given by the matrices of the form 
\[
\begin{pmatrix} 
0 & \cdots & \cdots & 0  \\
0 & \cdots & \cdots & 0 \\
\ast & \ast & \ast & \ast \\
\vdots & \cdots & \cdots & \vdots\\
\ast & \ast & \ast & \ast 
\end{pmatrix}. 
\]

Recall that, as at the beginning of \S\ref{sec:torus action}, we have a surjective map $\phi_{t\rightarrow 0}: X\rightarrow X^T\subset V^T$. %The next lemma is a direct computation and we omit the proof.  

\begin{remark}\label{rem:Jacobian at p in Sing}
Let $p\in V^T$ be the matrix $\begin{pmatrix} 
{\bf 0} \\
{\bf 0} \\
A_p \\
\end{pmatrix}$, 
where $A_p$ is a $(k-2)\times k$ matrix. Then the $k^2\times k^2$ Jacobian $J(Y)_p$ of $Y$ at $p$ has the following form 
\[
J(Y)_p=\bordermatrix{ & S_1 & S_2 & S_3 \cr
x_{1,h} & {\bf 0} & L_p & {\bf 0} \cr
x_{2,h} & L_p & {\bf 0} & {\bf 0} \cr 
x_{i,j} & {\bf 0} & {\bf 0} & {\bf 0}}. \qquad
\]
Here $S_1$ is the set of permanents that do not use the first row, $S_2$ is the set of permanents 
that do not use the second row, and $S_3$ is the set of permanents that use the first and second rows. 
Moreover, the $k\times k$ matrix $L_p=(\ell_{ij})$ is such that $\ell_{ij}$ is the $(k-2)\times (k-2)$ permanent 
of $A_p$ which does not use columns $i$ and $j$. Note that $L_p$ is symmetric with zeros on the main diagonal. 
\end{remark}

\begin{corollary}\label{cor:rank of tangent bundle}
With the same notation as in Remark \ref{rem:Jacobian at p in Sing}, whenever $p\in \phi_{t\rightarrow 0}(X)=X^T$ with $X\neq V^T$, the tangent space to $Y$ at $p$ is: 
\[
T_{Y,p} = V^T\oplus \ker(L_p)^{\oplus 2}, 
\]
where $T^0_{Y,p}=V^T$ and $T^1_{Y,p} = \ker(L_p)^{\oplus 2}$. In particular, 
$\dim_{\CC} T^1_{Y,p} = 2\dim_{\CC} \ker(L_p)$. 
\begin{proof}
The tangent space $T_{Y,p}$ is the kernel of the transpose of $J(Y)_p$. The two copies of $\ker(L_p)$ live
in the span of the first and second rows, respectively. 
\end{proof}
\end{corollary}

\begin{corollary}\label{cor:lower bound on codim}
Let $p\in X$ be any irreducible component of $Y=\mathrm{Sing}(P)$. Then the following upper bound holds:
\[
\dim X\leq \dim X^T + 2\dim_{\CC} \ker(L_p). 
\]
Equivalently, one has $\mathrm{codim}_V X\geq \mathrm{codim}_{V^T} X^T +2\mathrm{rk}(L_p)$. 
\begin{proof}
By Proposition \ref{bb in our case}, $X$ is contained in the closure of the vector bundle $T^1_{Y|X^T}$ over
$X^T$. Its rank is $2\dim_{\CC} \ker(L_p)$ by Corollary \ref{cor:rank of tangent bundle}. 
\end{proof}
\end{corollary}

A direct consequence of Corollary \ref{cor:lower bound on codim} is as follows. 

\begin{corollary}\label{cor: if rk at least 3, done}
Let $p\in X^T$ be general. If $\mathrm{rk}(L_p)\geq 3$, then $\mathrm{codim}_V X\geq 6$. 
\end{corollary}

\begin{lemma}\label{cor:rank of M is different from 1}
Let $p\in X^T$ be general. Then $\mathrm{rk}(L_p)\neq 1$. 
\begin{proof}
This is analogous to Proposition \ref{prop: no component corresponding to rank 1}. 
\end{proof}
\end{lemma}

\begin{proposition}\label{prop: case rank 2}
Let $k\geq 4$. Let $p\in X^T$ be general and suppose $\mathrm{rk}(L_p)=2$. Then:
\[
\mathrm{codim}_{V^T} X^T\geq 2. 
\]
\begin{proof}
By assumption all the $3\times 3$ minors of $L_p$ vanish. Any such principal submatrix $L$ has the form 
\[
L = \begin{pmatrix}
0 & a & b \\
a & 0 & c \\
b & c & 0 \\
\end{pmatrix}. 
\]
Hence $\det(L) = 2abc$. We regard the point $p$ as a $(k-2)\times k$ matrix. For any choice 
of columns $i_1,i_2,i_3$, there must be two indices $i_{n}, i_{m}$ such that the $(k-2)\times (k-2)$ permanent
not involving $i_n$ and $i_m$ is zero. Since $k\geq 4$, up to permuting columns, we may assume 
that $X^T$ is inside the locus $C$ defined by the vanishing of the $(k-2)\times (k-2)$ permanent $\mathrm{perm}_{12}$ not involving columns $1,2$ and of the $(k-2)\times (k-2)$ permanent $\mathrm{perm}_{34}$ not involving columns $3,4$. 
Since each of these permanents is irreducible by Corollary \ref{cor:perm is irreducible}, and since $\mathrm{perm}_{12}$ and $\mathrm{perm}_{34}$ are linearly independent, 
$C$ is a complete intersection of codimension two in $V^T$. Hence $\mathrm{codim}_{V^T} X^T\geq \mathrm{codim}_{V^T} C = 2$.
\end{proof}
\end{proposition}

\begin{theorem}\label{thm: case rank 0}
Let $k\geq 6$. Let $p\in X^T$ be general and suppose $\mathrm{rk}(L_p)=0$. Then one has
\[
\mathrm{codim}_{V^T} \phi_{t\rightarrow 0}(X)\geq 6.
\]
\begin{proof}
The assumption implies that $\phi_{t\rightarrow 0}(X)$ is inside the locus $C$ defined by the vanishing 
of all $(k-2)\times (k-2)$ permanents of any $(k-2)\times k$ matrix in $V^T$.

We have two cases: 
\begin{enumerate}
\item[(i)] For a general $p\in X^T$, all the $(k-3)\times (k-3)$ permanents vanish. 
Hence $X^T$ is inside a cone over an irreducible component 
of the locus $C$, given by the vanishing of all the $(k-3)\times (k-3)$ of a $(k-3)\times k$ generic matrix. 
In this case one has $\mathrm{codim}_{V^T} X^T\geq 6$, by Proposition \ref{prop: subcases of rank 0} below for $h=3$ and $\ell=k-3$. We postpone its proof because it involves a more technical analysis. 

\item[(ii)] For a general $p\in X^T$, there exists a $(k-3)\times (k-3)$ permanent 
that does not vanish at $p$. 

We claim that in this case the inequality in Corollary \ref{cor:lower bound on codim} is strict, i.e. 
$X$ is {\it strictly} contained in the closure of the vector bundle $T^1_{Y|X^T}$ over
$X^T$. Otherwise, if equality holds, then for a general $p\in X^T$ any extension $q$ to 
a $k\times k$ matrix satisfies $q\in X$. 
Put free variables $z_{ij}$ on the first two rows of $q$. Any $(k-1)\times (k-1)$ permanent of $q$ vanishes, 
because $q\in Y$. Consider a $(k-1)\times (k-1)$ permanent $\mathrm{perm}_{k-1,k-1}$ of $q$ involving the first two rows 
consisting of $z_{ij}$ and containing a $(k-3)\times (k-3)$ nonvanishing
subpermanent; the latter exists because of the assumption on $p$. The condition $\mathrm{perm}_{k-1,k-1}(q)=0$ gives 
a linear relation among the $2\times 2$ permanents of the $2\times k$ submatrix of $q$ whose entries are the $z_{ij}$. 
However, permanents of fixed arbitrary size of a generic matrix are linearly independent, by Lemma \ref{lem:perms are l.i.}. Therefore we reached a contradiction. 

Thus $\mathrm{codim}_V X\geq \mathrm{codim}_{V^T} X^T +1$. To conclude it is enough to show
that $\mathrm{codim}_{V^T} X^T\geq 5$. This is proven in Proposition \ref{prop: subcases of rank 0},
where $h=2$ and $\ell=k-2$. 
\end{enumerate}
This concludes the proof. 
\end{proof}
\end{theorem}

The previous proof relies on the following result, which in turn improves the easier lower bound of Proposition \ref{lower-upperboundscodim}. 

\begin{proposition}\label{prop: subcases of rank 0}
Let $M$ be a generic $\ell\times (\ell+h)$ complex matrix for $h\geq 1$ and let $V=\CC^{\ell\times (\ell+h)}$. 
Let $P_{\ell,\ell+h}$ be the variety defined by all the $\ell\times \ell$ permanents of $M$. Then 
$\mathrm{codim}_V P_{\ell,\ell+h}\geq h+3$ for $\ell\geq 3$.
\begin{proof}
The proof is by induction on $\ell$, with $\ell=3$ as base case, which is implied by Theorem \ref{codims}. Let $Y=P_{\ell,\ell+h}$. We fix the $T=\CC^{*}$-action scaling by $t$ the first row of $M$. Hence $V^T$ is a linear subspace
of $(\ell-1)\times (\ell+h)$ matrices. Let $X$ be an irreducible component of $Y$. As before, we have a surjective map $\phi_{t\rightarrow 0}: X\rightarrow X^T\subset V^T$. The Jacobian $J(Y)$ calculated at a point $p\in V^T$ has the form 
\[
J(Y)_p=\begin{pmatrix}
N_p \\
{\bf 0} \\
\end{pmatrix},
\]
where the entries of $N_p$ are $(\ell-1)\times (\ell-1)$ permanents of the $(\ell-1)\times (\ell+h)$ matrix $p$. Here 
the columns of $J(Y)_p$ are indexed by subsets of $\ell$ elements of the $\ell+h$ column set of $M$. 
The rows of $N_p$ correspond to the $\ell+h$ variables on the first row of $M$. From the description of the Jacobian, as in Corollary \ref{cor:lower bound on codim} and using Proposition \ref{bb in our case}, we find that $\mathrm{codim}_{V} X\geq \mathrm{codim}_{V^T} X^T + \mathrm{rk}(N_p)$. 

We claim that $\mathrm{rk}(N_p)\neq 1,\ldots, h$. Indeed, assume that there is a nonzero entry in $N_p$. 
This corresponds to a nonvanishing $(\ell-1)\times (\ell-1)$ permanent $\mathrm{perm}_{\ell-1,\ell-1}$. 
Up to permuting columns, we may assume that  $\mathrm{perm}_{\ell-1,\ell-1}$ involves the first $\ell-1$ columns
and set $\mathrm{perm}_{\ell-1,\ell-1}(p)=a\neq 0$. 

Consider the submatrix $N$ of $N_p$ with $\ell+h-(\ell-1) = h+1$ columns, each corresponding to a subset $\ell$ columns 
of $M$. Thus 
\[
N=\bordermatrix{ & \lbrace 1,\ldots, \ell\rbrace & \lbrace 1,\cdots, \ell-1, \ell+1\rbrace & \ldots & \lbrace 1,\ldots, \ell-1,\ell+h\rbrace \cr
x_{1,\ell} & a & 0 & \cdots & 0\cr
x_{1,\ell+1} & 0 & a & \cdots & 0 \cr 
\vdots & 0 & 0 & \ddots & 0 \cr
x_{1,\ell+h} & 0 & 0 & \cdots & a}, \qquad
\]
where the variables $x_{1,\ell+j}$ are the last $h+1$ variables in the first row of $M$. The matrix $N$ 
is a diagonal $(h+1)\times (h+1)$ matrix with the evaluated permanent $\mathrm{perm}_{\ell-1,\ell-1}(p)=a$ on the main diagonal. Hence $\det(N)\neq 0$ and so either $\mathrm{rk}(N_p)=0$ or $h+1\leq \mathrm{rk}(N_p)\leq \ell+h$. 

We shall be done if we prove that $\mathrm{rk}(N_p)\geq h+3$. To this aim, we have to deal with the cases: 
\begin{enumerate}
\item[(i)] $\mathrm{rk}(N_p)=0$;

\item[(ii)] $\mathrm{rk}(N_p)=h+1$;
\item[(iii)] $\mathrm{rk}(N_p)=h+2$. 
\end{enumerate}

Suppose (i) holds true. Then all the $(\ell-1)\times (\ell-1)$ permanents of the $(\ell-1)\times (\ell+h)$ matrix $p$ vanish. This implies that $X^T$ is inside an irreducible component of $P_{\ell-1,\ell-1+h}$. 
So, by induction on $\ell\geq 3$, we have $\mathrm{codim}_{V^T} X^T\geq h+3$. 

Suppose (ii) holds true. Then it is enough to find two $(h+2)\times (h+2)$ minors, without
common factors to prove that $\mathrm{codim}_{V^T} X^T\geq 2$. 
Consider the following $(h+2)\times (h+2)$ submatrix of $N_p$:
\begin{small}
\[
S=\bordermatrix{ & \lbrace 1,\ldots, \ell\rbrace & \lbrace 1,\cdots, \ell-1, \ell+1\rbrace & \ldots & \lbrace 1,\ldots, \ell-1,\ell+h\rbrace & \lbrace 1,\ldots,\ell-2,\ell,\ell+1\rbrace \cr
x_{1,\ell} & a & 0 & \cdots & 0 & b_2\cr
x_{1,\ell+1} & 0 & a & \cdots & 0 & b_1\cr 
\vdots & \vdots & \vdots & \ddots & \vdots  & \vdots \cr
x_{1,\ell+h} & 0 & 0 & \cdots & a & 0 \cr 
x_{1,\ell-1} &  b_1 & b_2 & \cdots & b_{\ell+1} & 0}, \qquad
\]
\end{small}
where $a$ is the permanent on columns $\lbrace 1,\ldots, \ell-1\rbrace$ of $p$, $b_1$
is the permanent on columns $\lbrace 1,\ldots, \ell-2,\ell\rbrace$ of $p$, and $b_2$
is the permanent on columns $\lbrace 1,\ldots, \ell-2, \ell+1\rbrace$ of $p$. 
Note that $\det(S) = -2a^{h}b_1b_2$. Since $\det(S)=0$ and $a\neq 0$, we have either $b_1=0$ or $b_2=0$. Picking a different minor from the one above, we find another irreducible vanishing permanent. 
Since $X^T$ must be contained in the vanishing of two irreducible 
and linearly independent permanents, we find that $\mathrm{codim}_{V^T} X^T\geq 2$. 

To conclude in case (iii), it is enough to find a point $p\in V^T$ and a $(h+3)\times (h+3)$ minor of $N_p$ that is nonzero. Indeed, then $\mathrm{codim}_{V^T} \phi_{t\rightarrow 0}(X)\geq 1$. %To do so,
%we may find a point $p\in V^T$ where such a minor is nonzero. 
Let $q\in V^T$ be a matrix of the form 
\[
q=\bordermatrix{ & 1 & 2 & \cdots & \ell-1 & \ell & \ell+1 & \ell+2 & \cdots \cr
& 1 & 0 & \cdots & 0 & 0 & 0 & 0 & 0 \cr
& 0 & 1 & \cdots & 0 & 0 & 0 & 0 & 0 \cr 
& 0 & 0 & \ddots & 0 & a & b & 0 & 0 \cr
& 0 & 0 & 0 & 1 & d & c & 1 & 0 }.\qquad
\]
Then consider the following $(h+3)\times (h+3)$ minor of $N_q$, the submatrix of $J(Y)_q$: 
\begin{tiny}
\[
Q=\bordermatrix{ & \lbrace 1,\ldots, \ell\rbrace & \lbrace 1,\cdots, \ell-1, \ell+1\rbrace & \ldots & \lbrace 1,\ldots, \ell-1,\ell+h\rbrace & \lbrace 1,\ldots,\ell-2,\ell,\ell+1\rbrace & \lbrace 1,\ldots,\ell-2,\ell,\ell+2\rbrace \cr
x_{1,\ell-2} & a & d & \cdots & 0 & ac+bd & a\cr
x_{1,\ell-1} & d & c & \cdots & 1 & 0 & 1 \cr 
x_{1,\ell} & 1 & 0 & \cdots & 0  &  c & 0 \cr
x_{1,\ell+1} & 0 & 1 & \cdots & 0 & d & 0 \cr 
x_{1,\ell+2} & 0 & 0  & 1 & 0 & 0 & d \cr
\vdots &  \vdots & \vdots & \vdots & \vdots & \vdots & \vdots \cr
x_{1,\ell+h} & 0 & 0 & 0 & 1 & 0 & 0}. 
\]
\end{tiny}
The lower-left $(h+1)\times (h+1)$ corner is the identity matrix. The matrix $Q$ is divided into
two linearly independent blocks: an $(h-2)\times (h-2)$ identity matrix (inside 
the lower-left $(h+1)\times (h+1)$ identity corner) and the following $5\times 5$ matrix 
\[
Q' = \begin{pmatrix}
a & d & 0 & ac+bd & a \\
d & c & 1 & 0 & 1 \\
1 & 0 & 0 & c & 0 \\
0 & 1 & 0 & d & 0 \\
0 & 0 & 1 & 0 & d\\ 
\end{pmatrix}.
\]
Hence $\mathrm{rk}(Q) = h-2 + \mathrm{rk}(Q')$. Now $\det(Q') = d(d^2-db+2ac-d+b)$, which is nonzero for generic choices of $a,b,c,d$. For such 
choices, $\mathrm{rk}(Q')=5$ and the proof is complete. 
\end{proof}
\end{proposition}

We are ready to improve von zur Gathen's Theorem \ref{vzgthm}. 

\begin{theorem}\label{codim of sing loc}
Let $k\geq 6$. The singular locus $\mathrm{Sing}(P) = \lbrace \mathrm{prk}(M)\leq k-2\rbrace$ has codimension between $6$ and $2k$. 
\begin{proof}
By Corollary \ref{cor:lower bound on codim}, it is enough to show that for any irreducible component $X$ of $Y=\mathrm{Sing}(P)$, we have 
\[
\mathrm{codim}_{V^T}X^T+2\mathrm{rk}(N_p)\geq 6. 
\]
By Corollary \ref{cor: if rk at least 3, done} and Corollary \ref{cor:rank of M is different from 1}, we have 
two cases to deal with: either $\mathrm{rk}(N_p)=2$ or $\mathrm{rk}(N_p)=0$. The first case is achieved by Proposition \ref{prop: case rank 2}. The second case is implied by Theorem \ref{thm: case rank 0}. 
\end{proof}
\end{theorem}

\begin{definition}
For any subset $R\subset [k]$ of rows, let $J_R$ be the ideal generated by all the $|R|\times |R|$ permanents of the $|R|\times k$ submatrix $M_{R,[k]}$, i.e. 
the submatrix of $M$ whose rows are indexed by $R$. For any subset of columns $C\subset [k]$, one similarly defines $J_C$. 
\end{definition}

We omit the proof of the following straightforward result.

\begin{lemma}
For any partition of rows $R_1\sqcup R_2 = [k]$ or of columns $C_1\sqcup C_2 = [k]$, we have the inclusions of ideals 
\[
I(\mathrm{Sing}(P)) \subset J_{R_1} + J_{R_2} \mbox{and } I(\mathrm{Sing}(P)) \subset J_{C_1} + J_{C_2}. 
\]
\end{lemma}

\begin{corollary}
Suppose Conjecture \ref{conj: ci} holds. Then any irreducible component $C$ of $\mathrm{Sing}(P)$ whose prime ideal $I(C)$ contains $J_{R_1} + J_{R_2}$ for a partition $R_1\sqcup R_2 = [k]$ has codimension $\geq 2k$. 
\begin{proof}
By Theorem \ref{thm:codimP_kn}, the assumption implies that the codimension of $J_{R_1} + J_{R_2}$ is $2k$. Since $I(C)$ contains $J_{R_1} + J_{R_2}$, the statement follows. 
\end{proof}
\end{corollary}

\begin{conj}\label{radequality}
We have the following equality of radical ideals: 
\begin{equation}\label{eqradequality}
\mathrm{rad}(I(\mathrm{Sing}(P)) = \bigcap_{(S_1,S_2)\in \Pi} \mathrm{rad}(J_{S_1}+J_{S_2}),
\end{equation}
where $\Pi$ is the set of partitions $(S_1, S_2)$ of the $k$ rows or the $k$ columns. 
\end{conj}

\begin{remark}
If both Conjectures \ref{radequality} and \ref{conj: ci} were true for each $k$, the codimension of $\mathrm{Sing}(P)$ would be $2k$, i.e. the upper bound in von zur Gathen's Theorem \ref{vzgthm} would be sharp for each $k$. Equality \eqref{eqradequality} has been computationally checked for $k=3$ in \texttt{Macaulay2}. We do not know whether it is true even for $k=4$. 
\end{remark}

%\newpage 
\section{Code}\label{codes}
The majority of the following code, which simplifies that in an earlier version of this paper, was provided by an anonymous reviewer, to whom we are grateful.
\begin{mybox}
{\color{blue}
\begin{tiny}
\begin{verbatim}
K = QQ;
-- k x (k+1) matrices 
k = 4 -- or k = 5 
R = K[x_(1,1)..x_(k,k+1)];
M = matrix for i in 1..k list for j in 1..k+1 list x_(i,j);
P = permanents(k,M);
B1 = diff(matrix{{x_(1,1)..x_(1,k+1)}}, transpose gens P);
v = flatten entries transpose M_{1..k}^{1..k-1};

-- random A (for k = 4)
A = random(K^(k-1),K^(#v));

-- special A (for k = 5)
A = matrix {{3, 3, 2, 1, -1, 0, -3, 3, 2, -3, 2, 0, -3, 2, 3, -2, 2, 2, -3, -3},
{-2, -2, -1, 1, -1, 0, -2, -2, -1, -3, 2, -2, -1, 3, -2, -2, 2, -1, -1, -1},
{-2, -2, 1, 2, 3, 0, 0, -3, 2, 2, -3, -3, -1, 2, -3, 2, -2, 3, -2, 2},
{-3, 0, -3, -1, 1, 2, -1, 2, -3, 2, 1, 0, -3, -1, -1, -3, -2, 3, -1, -3}};

F = first entries (matrix{{x_(2,1)..x_(k,1)}}*A);
L = apply(#v, i-> v#i => F#i);
BB = sub(B1, L);


-- k = 4 case
P = det(BB); 
S = K[x_(2,1),x_(3,1),x_(4,1)];
PP = sub(P,S);
Sing = ideal diff(vars S,PP);
time codim Sing

use R
J = time minors(4,BB);
gbTrace=1
time codim J

-- k = 5 case 
J = time minors(3,BB);
gbTrace=1
time codim J
\end{verbatim}
\end{tiny}
}
\end{mybox}

\end{document}